\documentclass[12pt]{amsart}
\usepackage{amssymb,amsmath,amsthm,comment}
\usepackage{pdfpages,graphicx} 
\usepackage[section]{placeins} 
\usepackage{wrapfig}
\usepackage{float}
\usepackage{lineno}

\usepackage{setspace}
\newcommand{\shrinkmargins}[1]{
  \addtolength{\textheight}{#1\topmargin}
  \addtolength{\textheight}{#1\topmargin}
  \addtolength{\textwidth}{#1\oddsidemargin}
  \addtolength{\textwidth}{#1\evensidemargin}
  \addtolength{\topmargin}{-#1\topmargin}
  \addtolength{\oddsidemargin}{-#1\oddsidemargin}
  \addtolength{\evensidemargin}{-#1\evensidemargin}
  }
\shrinkmargins{0.5}

\newtheorem{theorem}{Theorem}
\newtheorem{lemma}[theorem]{Lemma}

\newtheorem{corollary}[theorem]{Corollary}
\newtheorem*{theorem*}{Theorem}

\newtheorem{proposition}[theorem]{Proposition}
{Claim}

\theoremstyle{definition}
\newtheorem*{definition}{Definition}

\theoremstyle{remark}

\newtheorem*{remark}{Remark}
\newtheorem*{remarks}{{\bf Remarks}}

\numberwithin{theorem}{section} \numberwithin{equation}{section}

\usepackage{caption}
\usepackage{multirow}

\def\func#1{\mathop{\rm #1}}%
\newcommand{\norm}[2][]{\ensuremath{\left|\!\left|#2\right|\!\right|_{#1}}}

\newcommand{\floor}[1]{\left\lfloor #1 \right\rfloor}

\begin{document}
\title[Asymptotic Normality]{Asymptotic Normality of the Coefficients of the Morgan-Voyce Polynomials}
\author{Moussa Benoumhani}
\address{Department of Mathematics, University of M'Sila, M'Sila, Algeria}
\email{benoumhani@yahoo.com}
\author{Bernhard Heim }
\address{Lehrstuhl A f\"{u}r Mathematik, RWTH Aachen University, 52056 Aachen, Germany}
\email{bernhard.heim@rwth-aachen.de}
\author{Markus Neuhauser}
\address{Kutaisi International University, 5/7, Youth Avenue,  Kutaisi, 4600 Georgia}
\email{markus.neuhauser@kiu.edu.ge}
\begin{abstract}
We study arithmetic and asymptotic properties of polynomials provided by
$Q_n(x):= x \sum_{k=1}^n  k \, Q_{n-k}(x)$ with initial value $Q_0(x)=1$.
The coefficients satisfy a central limit theorem and a local limit theorem
involving Fibonacci numbers. We apply methods of Berry and Esseen, Harper,
Bender, and Canfield.
\end{abstract}
\maketitle
\section{Introduction and main results}
Let $g$ be a normalized arithmetic function. Let $G(t):= \sum_{n=1}^{\infty} g(n) \, t^n$ be
regular at $t=0$. We are interested in asymptotic properties of 
the double sequence $a(n,k)$ of coefficients of the polynomials $p_n(x)$ defined by
\begin{equation}\label{Q}
\sum_{n=0}^{\infty} p_n(x) \, t^n = \frac{1}{ 1- x \,  G(t)}.
\end{equation}
Let $g(n)= \frac{1}{n!}$. Then
the coefficients are asymptotically normal \cite{Ha67, Be73}. 
They calculate the number of partitions of an
$n$-set having exactly $k$ labeled blocks. Thus, are equal to $k! \, S_{n,k}$, where $S_{n,k}$ are
the Stirling numbers of the second kind (\cite{Be73}, section 3 applications).
It is obvious that the case
$g(n)=1$ leads to the Central Limit Theorem 
by de Moivre--Laplace, 
since $$a(n,k)= \binom{n-1}{k-1}$$ (cf.\ \cite{HNT20}).

In this paper we study the case $g(n)=n$ in detail and show that the
underlying expected values and variances involve Fibonacci numbers.

We denote the associated polynomials and their coefficients by
$$Q_n(x)= \sum_{k=1}^n A_{n,k} \, x^k.$$
Let $B_n(x)$ denote the Morgan-Voyce polynomials (\cite{Ko01}, Chapter 41). Then $Q_{n+1}(x) = x \, B_n(x)$.
In 1959, Morgan-Voyce \cite{MV59} discovered importance of $B_n(x)$ in the study of electric ladder networks of resistors.
Further, in 1967 and 1968,
several papers appeared by Swamy, Basin, Hoggatt, Jr., and Bicknell.
These polynomials are closely related to Fibonacci polynomials.
From \cite{HNT20}, we know that
\begin{equation*}
A_{n,k}:=  \binom{n+
k-1}{2k-1}, \qquad n \in \mathbb{N}, \, 0 \leq k \leq n.
\end{equation*}
The definition (\ref{Q}) is equivalent to $Q_n(x)= x \sum_{k=1}^n k Q_{n-k}(x)$ with initial value $Q_0(x)=1$.
This hereditary recurrence relation
can be reduced to a three term recurrence 
relation
\begin{equation}\label{3term}
Q_{n+2}(x)-(2+x)Q_{n+1}(x) + Q_n(x)=0, \qquad n \geq 0,
\end{equation}
with initial values 
$Q_1(x)=x$ and $Q_{2}=\left( x+2\right) x$. 
The double sequence $A_{n,k}$ (see Table~\ref{Table1}) is recorded in Sloane's data base 
as A078812.

\begin{table}[H]
\[
\begin{array}{r|rrrrrrrrrrrr}
\hline
n\backslash k & 1 & 2 & 3 & 4 & 5 & 6 & 7 & 8 \\ \hline \hline
1 & 1 & & & & & & &  \\
2 & 2 & 1 & & & & & & \\
3 & 3 & 4 & 1 & & & & & \\
4 & 4 & 10 & 6 & 1 & & & & \\
5 & 5 & 20 & 21 & 8 & 1 & &  \\
6 & 6 & 35 & 56 & 36 & 10 & 1 & & \\
7 & 7 & 56 & 126 & 120 & 55 & 12 & 1 &  \\
8 & 8 & 84 & 252 & 330 & 220 & 78 & 14 & 1  \\
\hline
\end{array}
\]
\caption{\label{Table1} Coefficients $A_{n,k}$ of $Q_n(x)$.}
\end{table}

\subsection{Asymptotic normality by singularity analysis}
Our first result, using techniques from singularity analysis 
(see Canfield \cite{Ca15}, Section 3.6), states
that the double sequence $A_{n,k}$ is asymptotically
normal with asymptotic mean $a_n=
\frac{1}{\sqrt{5}} \,\, n $ and asymptotic variance
$ b_n^2 =  \frac{2}{5} \, \frac{1}{\sqrt{5}}\,\, n$.
\begin{theorem}\label{th1}
There exist real sequences $(a_n)_n$ and $(b_n)_n$ with $b_n >0$, such that
\begin{equation}\label{singularity}
\lim_{n \to \infty}  \func{sup}_{x \in \mathbb{R}} \left|
{ \frac{1}{\sum_{k=0}^n A_{n,k}}
\sum_{ k \leq a_n  + x b_n} A_{n,k} - \frac{1}{\sqrt{2 \pi}} \int_{-\infty}^x
e^{-\frac{t^2}{2}} \, dt} \right|=0.
\end{equation}
The sequences given by $a_n:= n \, a$ and $b_n^2:= n \, b^2$ satisfy  (\ref{singularity})
with
\begin{equation*}
a:= \frac{1}{\sqrt{5}} \text{ and } b^2:= \frac{2}{5} \, \frac{1}{\sqrt{5}}.
\end{equation*}
\end{theorem}
Obviously, the sequences are not unique.
In singularity analysis they are frequently 
called mean and variance (see Bender \cite{Be73}) to
indicate the link to probability theory and 
the classical Central Limit Theorem (cf.\ \cite{Fi11}, section 1.1).

In this paper, we provide sequences $(\mu_n)_n$ and
$\left( \sigma _{n}^{2}\right) _{n}$, the expected value and variance
of a suitable sequence of random variable $X_n$. We obtain a refined version of Theorem \ref{th1}.
We apply a method of Harper \cite{Ha67} (see also \cite{Ca15} section 3.4: the method of negative roots).
The values $\mu_n$ and $\sigma_n^2$ involve Fibonacci numbers $F_n$ and the golden ratio $\varphi$.
They characterize the peaks and modes of the coefficients of the
generating series. These are real-rooted polynomials, which 
are related to the Jonqui\`{e}re function (also
called polylogarithm $\func{Li}_{-1}(t)$) given by
$
{G}(t):= \sum_{n=1}^{\infty} n \, t^n$. The radius of convergence is $R=1$.

\subsection{Probabilistic approach}
Let $X$ be a random variable. Then $X$ is called normally distributed, if the probability
$P( X \leq x)$ is equal to the normal distribution
\begin{equation*}
\Phi(x):=\frac{1}{ \sqrt{2 \pi} } \int_{- \infty}^x e^{-\frac{t^2}{2}} \, dt,
\end{equation*}
for all real $x$. The distribution $\Phi$ is called normal or 
Gauss distribution. Let further, $\phi(x)$ be the underlying density function.
\begin{definition}
A sequence $(S_n)_n$ of random variables satisfies a Central Limit Theorem, if 
there exist sequences $(a_n)_n$ and $(b_n)_n$ with $a_n, b_n \in \mathbb{R}$ and $b_n >0$, such that
the normalized random variables
\begin{equation*}
S_n^{*} := \frac{S_n - a_n}{b_n},
\end{equation*}
converge in distribution against the normal distribution 
\begin{equation*}
P(S_n^{*} \leq x) \stackrel{
{D}}{\longrightarrow} 
\Phi(x).
\end{equation*}
\end{definition}
Consider the sequence of random variables $X_n \in \{0,1, \ldots, n\}$
defined by: 
\begin{equation}\label{X}
P(X_n=k):= \frac{A_{n,k}}{\sum_{m=0}^n A_{n,m}}, \text{ where }Q_n(x) = \sum_{k=0}^n A_{n,k} \, x^k.
\end{equation}
Let $S_n:= \sum_{m=1}^n X_m$. Then 
\begin{equation} \label{S}
P(S_n^{*} \leq x) = \frac{1}{ \sum_{m=0}^n A_{n,m}}               
\sum_{k \leq a_n + x \, b_n} A_{n,k},
\end{equation}
leads to a probabilistic interpretation of Theorem \ref{th1}.
\begin{theorem}\label{CLT} 
Let $S_n^{*}$ be the normalized random variables as defined in (\ref{S})
with $a_n$ and $b_n^2$ provided by the expected value $\mathbb{E}(X_n)$ and variance
$\mathbb{V}(X_n)$ of the random variable $X_n$,
associated with the double sequence $A_{n,k}$.
Then $P(S_n^{*} \leq x) \stackrel{
{D}}{\longrightarrow} 
\Phi(x)$. Therefore,
\begin{equation}\label{CLTformula}
\lim_{n \to \infty} {\norm{
\frac{1}{\sum_{m=0}^n A_{n,m}} \sum_{k \leq \mathbb{E}(S_n)
 + \sqrt{\mathbb{V}(S_n)} \,  x} A_{n,k} - 
\Phi(x) }}_{
\mathbb{R}} =0.
\end{equation}
\end{theorem}
Here ${\norm{f(x)}}_M$ denotes the supremum norm of $f$ on $M\subset \mathbb{R}$.
This theorem is obtained by a method introduced by Harper \cite{Ha67, Ca15} and
a result 
by P\'olya for continuous distribution functions.
We determine the explicit values of the mean and variance of $X_n$.
These are expressed in terms of Fibonacci numbers $F_n$ and are 
related to the peaks and plateaus of the unimodal sequence
\begin{equation*}
A_{n,0} \leq A_{n,1} \leq \ldots \leq A_{n,m} \geq \ldots \geq  A_{n,n}
\end{equation*}
by Darroch's Theorem \cite{Da64}.
The sequence is log-concave and therefore, unimodal, since the 
polynomials $Q_n(x)$ are orthogonal polynomials.
Indeed, $Q_n(x)= x \, U_{n-1}(x/2+1)$, where $U_n(x)$ are the Chebyshev polynomials
of the second kind 
\cite{HNT20}.
We first determine the normalizing factor, the average over all coefficients.
\begin{lemma} \label{lem1}
The normalizing factor is related to Fibonacci numbers. 
We have 
\begin{equation*}
\sum_{k=0}^n A_{n,k} = 
Q_n(1) = F_{2n}.
\end{equation*}
\end{lemma}

\begin{proposition} \label{prop1}
The random variable $X_n$ defined in (\ref{X}) 
has the expected value $\mu_n:= \mathbb{E}(X_n)$ and variance $\sigma_n^2:= \mathbb{V}(X_n)$
given by
\begin{eqnarray*}
\mu_n & = & \frac{2}{5}
\left( \frac{F_{2n+1}}{F_{2n}} - \frac{1}{2} + \frac{1}{n}\right) 
 \, n, \\
\sigma_n^2 & = &
\frac{4}{25} \, \left( \frac{F_{2n+1}}{F_{2n}} -  \frac{1}{2}- \frac{n}{F_{2n}^2} - \frac{1}{2n}\right) \, n.
\end{eqnarray*}
\end{proposition}
\begin{remarks} \ \\
a) Note that $\mu _{n}$ is never an integer for
$n\geq 2$. If $\mu _{n}$ is an integer then
$2\frac{F_{2n+1}}{F_{2n}}n$ is also an integer. Since
$F_{2n+1}$ and $F_{2n}$ are coprime, 
$\frac{2n}{F_{2n}}$ is also an integer. Since
$F_{2n}>2n$ for $n\geq 3$ and  the only remaining case
is $n=1$. 
\\
b) The reciprocal polynomial of $Q_n(x)$ has the coefficients $\binom{N - k-1}{k}$ with $N=2n$.
These sequences had been investigated by Tanny and Zucker \cite{TZ74, TZ78} and Benoumhani \cite{Be03}.
It had been shown that the sequence is strictly log-concave, the smallest mode had been determined and
the indices, at which a double maximum occurs.
\end{remarks} 
Darroch's theorem \cite{Da64} (see also Benoumhani \cite{Be96}) leads to:
\begin{corollary} 
The modes of $Q_n(x)$ are located around $\mu_n$. Let $m_n$ be a mode, then
\begin{equation*}
0 <
\left| 
\frac{2}{5}
\left( \frac{F_{2n+1}}{F_{2n}} - \frac{1}{2} + \frac{1}{n}  \right) 
 \, n - m_n 
\right| < 1. 
\end{equation*}
\end{corollary}
Using elementary calculations and explicit solutions of 
the Pell--Fermat equation leads to:
\begin{theorem}\label{exact}
Let $n \in \mathbb{N}$.
The polynomials $Q_n(x)$ are unimodal and have at most two modes.
The smallest mode $m
$ is uniquely determined by
$$\frac{\sqrt{5n^{2}+1}- 1}{5} \leq m
< \frac{\sqrt{5n^{2}+1} +4}{5}.$$
The mode
is unique if $5m^{2}+2m\neq n^{2}$. We have two modes $m_{k}
$ and $m_{k}
+1$, iff
\begin{equation*}
\left( 
\begin{array}{c}
5m_k +1 \\ n_k
\end{array}\right)
= \left(
\begin{array}{cc}
161 & 360 \\
72 & 161
\end{array}
\right)^{k}\left(
\begin{array}{c}
1 \\
0
\end{array}\right)
\qquad \text{ for } k \geq 1.
\end{equation*}
\end{theorem}

\begin{remark}
If $5m^{2}+2m=n^{2}$ has
positive integer solutions, there
are two modes of
$
A_{n,m}
$ at $m
$ and $m
+1$.
For example, for $n=n_1=72$ and $m=m_1=32$, we have
$\binom{103}{63}=61218182743304701891431482520=\binom{104}{65}$. The sequence of
these $n_k$ starts with
\[
72,23184,7465176,2403763488,774004377960,249227005939632,\ldots
\]
and the corresponding $m_k$ are
\[
32,10368,3338528,1074995712,346145280800,111457705421952,\ldots
\]
\end{remark}

The expected values $\mu_n$ and variances $\sigma_n^2$ converge against
the sequences $a_n$ and $b_n$,
obtained in Theorem \ref{th1} by methods of singularity analysis.
\begin{corollary} \label{limits}
Let $X_n$ be the random variable defined by (\ref{X}) with expected value $\mu_n$ and
variance $\sigma_n^2$. Let $a$ and $b$ be as in Theorem \ref{th1}. Then
\begin{eqnarray*}
\lim_{n \to \infty} \frac{\mu_n}{n} & = &  \frac{2}{5} \left( \varphi - \frac{1}{2} \right) = \frac{1}{\sqrt{5}}=a, \\
\lim_{n \to \infty} \frac{\sigma_n^2}{n} & = & \frac{4}{25} \left( \varphi - \frac{1}{2} \right) =
\frac{2}{5} \, \frac{1}{\sqrt{5}} = b^2,        
\end{eqnarray*}
where $\varphi$ denotes the golden ratio.
\end{corollary}
This follows from Proposition \ref{prop1}, since we know that
\begin{equation*}
\lim_{n \to \infty} \frac{F_{2n+1}}{F_{2n}} = \varphi \qquad (\text{Kepler}).
\end{equation*}
It is possible to obtain a rate of convergence of (\ref{CLTformula})
by utilizing the Berry--Esseen theorem (see for example \cite{Ca15}, section 3.2),
which implies Theorem \ref{CLT}, since the variance satisfies $$\lim_{n \to \infty} \sigma_n 
= \infty .$$
\begin{theorem}\label{BEapplication}
Let $S_n^{*}$ be the normalized random variables as defined in (\ref{S}),
associated with $A_{n,k}$.
Then $$ \vert P(S_n^{*} \leq x) - \Phi(x) \vert \, \leq \, C \, \sigma_n^{-1}.$$ 
Here $C>0$ can be chosen as $C=0.7975$.
\end{theorem}
The Central Limit Theorem, to quote Bender \cite{Be73} and Canfield \cite{Ca75}, provides a certain
qualitative feel for the numbers $A_{n,k}$. Further information is provided by
local limit theorems (cf.\ \cite{Ca15}, section 3.7).
Let $A_{n,k}^{*}:= A_{n,k}/Q_n(1)$.
We consider subsets $K \subset \{0, 1, \ldots, n\}$  and $k \in K$ with the asymptotic behavior
\begin{equation*}
{A_{n,k}^{*}} \,\, \sim   \,\,   \left.\frac{\phi(x)}{\sigma_n}\right\vert_{x= (k-\mu_n)/\sigma_n}.
\end{equation*}
More generally, we say the doubly indexed sequence $a(n,k)$ satisfies a local limit theorem on a set $S$ of real numbers
provided
\begin{equation*}\label{canfield:local}
\func{sup}_{x \in S} \Big| \frac{ \sigma_n \, 
a(n,\floor{ \mu_n + x \, \sigma_n})}{\sum_{k} a(n,k)} - \phi(x) \Big| \longrightarrow 0,
\end{equation*}
where $\mu_n$ and $\sigma_n$ are the expected values and variances (\cite{Ca15}, definition 3.7.1).
Further, we quote the following result.
\begin{theorem}[Bender, \cite{Ca15}, Theorem 3.7.2] \label{theorem:local}
Suppose that $a(n,k)$ are asymptotically normal and $\sigma_n^2 \rightarrow \infty$. If for
each $n$ the sequence $a(n,k)$ is unimodal in $k$, then $a(n,k)$ satisfies a local limit theorem on the set
$\{ x \, \, : |x| \geq \varepsilon\}$, for any $\varepsilon >0$. If for each $n$ the sequence $a(n,k)$ is log-concave in $k$, the
$a(n,k)$ satisfies a local limit theorem on the set $\mathbb{R}$.
\end{theorem}

We have that $A(n,k)$ are asymptotically normal by Theorem \ref{CLT} and $\sigma_n^2 \rightarrow \infty$.
Further, that $Q_n(x)$ has real roots. This implies by Theorem \ref{theorem:local}:
\begin{corollary} \label{Anwendung}
\begin{equation*}
\lim_{n \to \infty} 
{\norm{\sigma_n \, 
{A_{n,
\floor{\mu_n + x \, \sigma_n}}^{*}} - \phi(x)
}}_{
\mathbb{R}} = 0.
\end{equation*}
\end{corollary}
\begin{corollary} \label{local}
Further, we have the asymptotic formula
\begin{equation*}
A_{
n,k
} \,\, \sim \,\, F_{2n} \frac{ \phi(x)}{\sigma_n
}  \text{ for } n \to \infty,
\end{equation*}
where $k= \mu_n + x \sigma_n$ and $x$ is bounded.
\end{corollary}

\section{Singularity analysis}
Bender \cite{Be73} provided a criterion for a double sequence $a
\left( {n,k}\right) $ to be
asymptotically normal.
Typically the criterion applies, if
the generating function has only one singularity on the circle of convergence. 
We recall the approach 
by Bender (see also \cite{Ca15}, section 3.6 Method 4).

\begin{theorem}[Bender]
Let $f(x,t)= \sum_{k,n} \, a(n,k) \, x^k \, t^n$, with $a(n,k)\geq 0$.
Suppose there exist
\begin{itemize}
\item[(i)] a function $A(s)$ continuous and non-zero near $0$,
\item[(ii)] a function $r(s)$ with bounded third derivative near $0$,
\item[(iii)] a non-negative integer $m$, and
\item[(iv)] positive numbers $\varepsilon$ and $\delta$ such that
\begin{equation*}
\left( 1 - \frac{t}{r(s)}\right)^m \, f\left( e^s,t\right) - \frac{A(s)}{1-\frac{t}{r(s)}}
\end{equation*}
is analytic and bounded for $\vert s \vert < \varepsilon$, $\vert t \vert < r(0) + \delta$.
\end{itemize}
Put $a= - r'(0)/r(0)$ and $b^2= a^2 - r''(0)/r(0)$. If $b^2 \neq 0$, then the numbers $a(n,k)$
are asymptotically normal with $a_n= n \, a$ and $b_n^2 = b^2 \, n$.
\end{theorem}
\subsection{Proof of Theorem \ref{th1}}
Let $a(n,k)=A_{n,k}$. Let $G(t)= \sum_{n=1}^{\infty} n \, t^n$. This power series
has radius of convergence 
$1$ and has a pole at $t=1$.
The generating series of $A_{n,k}$ (we refer to \cite{HNT20}) is provided by
\begin{equation*}
f(x,t)= \frac{1}{1-x\, G(t)}= \sum_{n=0}^{\infty} Q_n(x) \, t^{n}.
\end{equation*}
We have $G(t)= \frac{t}{(t-1)^2}$, thus
$G(t)=1$ has two real solutions. 
The smallest in absolute value is given by 
$t_1= \frac{3- \sqrt{5}}{2}$. 

Suppose $x\neq 0$. If $G\left( t\right) =\frac{1}{x}$, then
$xt=\left( 1-t\right) ^{2}$ i.~e.\
$t^{2}-\left( 2+x\right) t+1=0$. Therefore,
$t=1+\frac{x}{2}-\sqrt{\frac{x^{2}}{4}+x}$ with the principle
branch of the square root for $\left| x-1\right| <1$.
Let
$r\left( s\right) =1+\frac{e^{
s}}{2}-\sqrt{\frac{e^{
2s}}{4}+e^{
s}}$
for $\left| s\right| <\ln 2$.
Then the pole of $f\left(
x,t\right) $ closest to $0$ is
located at $r\left( s\right) $. The other pole is located
at
$1/r\left( s\right) =1+\frac{e^{
s}}{2}+\sqrt{\frac{e^{
2s}}{4}+e^{
s}}$.

Since
\begin{eqnarray*}
f\left( e^{s}
,t\right) &=&\frac{1}{1-e^{s}
G\left( t
\right) }=\frac{\left( 1-
t
\right) ^{2}}{1-\left( 2+
e^{s}\right)
t+t
^{2}}=1+\frac{
e^{s}t
}{\left( 1-r\left( s\right) t
\right) \left( 1-t/r\left( s\right)
\right) }\\
&=&1+\frac{e^{s}}{2
\sqrt{\frac{e^{
2s}}{4}+
e^{
s}}}\left(
\frac{1}{1-t/r\left( s\right)
}-
\frac{1}{1-
r\left( s\right)
t}\right),
\end{eqnarray*}
we obtain
\begin{equation}
f\left( 
e^{s},t\right) -
\frac{e^{s}}{2\sqrt{\frac{e^{
2s}}{4}+e^{
s}}}\frac{1
}{1-t/r\left( s\right) }=1-
\frac{e^{s}}{2\sqrt{\frac{e^{
2s}}{4}+e^{
s}}}
\frac{1}{
1-
r\left( s\right) t}.
\label{eq:analytisch}
\end{equation}
For $\left| s\right| <\ln 2$ holds
$\left| r\left( s\right) \right| <\frac{1}{2}$. Therefore,
$\left| 1/r\left( s\right) \right| >2$ and
(\ref{eq:analytisch}) is
analytic and bounded for $\left| t\right| <\frac{3}{2}$.

We obtain
$r^{\prime }\left( s\right) =\frac{e^{
s}}{2}-\frac{1}{2}\left( \frac{e^{
2s}}{4}+e^{
s}\right) ^{-1/2}\left( \frac{e^{
2s}}{2}
+e^{
s}\right) $ and 
$$r^{\prime \prime }\left( s\right) =\frac{e^{
s}}{2}+\frac{1}{4}\left( \frac{e^{
2s}}{4}+e^{
s}\right) ^{-3/2}\left( \frac{e^{
2s}}{2}
+e^{
s}\right) ^{2}-\frac{1}{2}\left( \frac{e^{
2s}}{4}+e^{
s}\right) ^{-1/2}\left( e^{
2s}+e^{
s}\right). $$
Further,
\begin{eqnarray*}
r^{\prime \prime \prime }\left( s\right) &=&\frac{e^{
s}}{2}-\frac{3}{8}\left( \frac{e^{
2s}}{4}+e^{
s}\right) ^{-5/2}\left(
\frac{e^{
2s}}{2}
+e^{
s}\right) ^{3}\\
&&{}+\frac{1}{2}\left( \frac{e^{
2s}}{4}+e^{
s}\right) ^{-3/2}\left(
\frac{e^{
2s}}{2}
+e^{
s}\right) \left( e^{
2s}+e^{
s}\right) \\
&&{}+\frac{1}{4}\left( \frac{e^{
2s}}{4}+e^{
s}\right) ^{-3/2}\left( e^{
2s}+e^{
s}\right) ^{2}\\
&&{}-\frac{1}{2}\left( \frac{e^{
2s}}{4}+e^{
s}\right) ^{-1/2}\left(
2e^{
2s}+
e^{
s}\right),
\end{eqnarray*}
which shows that the third derivative is bounded near $0$,
$r^{\prime }\left( 0\right) =
\frac{1}{2}-
\frac{1}{2}\frac{2}{\sqrt{5}}\frac{3}{2}=
\frac{1}{2}-
\frac{3}{2\sqrt{5}}$, and
$r^{\prime \prime }\left( 0\right) =\frac{1}{2}+\frac{1}{4}\frac{8}{5\sqrt{5}}\frac{9}{4}-\frac{2}{\sqrt{5}}=\frac{1}{2}-\frac{11}{10\sqrt{5}}$.
With
$r\left( 0\right) =\frac{3}{2}-\frac{\sqrt{5}}{2}=\left( \frac{3}{2\sqrt{5}}-\frac{1}{2}\right) \sqrt{5}$. This yields
$a= \frac{1}{\sqrt{5}}$.
As $1/r\left( 0\right) =\frac{3}{2}+\frac{\sqrt{5}}{2}$, we obtain
$$b=\frac{1}{5}-\left( \frac{3}{2}+\frac{\sqrt{5}}{2}\right) \left( \frac{1}{2}-\frac{11}{10\sqrt{5}}\right) =\frac{1}{5}-\frac{3}{4}
+\frac{11}{20}-
\frac{\sqrt{5}}{4}+\frac{33}{20\sqrt{5}}=\frac{2}{5\sqrt{5}}.$$
\section{The Berry--Esseen theorem}
Let $X \in \{0,1, \ldots, n\}$ be a random variable 
with expected value $\mathbb{E}(X)$ and variance $\mathbb{V}(X)$ for $n \in \mathbb{N}$.
Let $P_{X} (x)= \sum_{k=0}^{n}
P(X=k) \, x^{k}
$ be the probability generating function.
A straightforward calculation leads to
\begin{eqnarray*}
\mathbb{E}(X) & = & P'(1), \\
\mathbb{V}(X) & = & P''(1) - \left( P'(1)\right)^2 + P'(1).
\end{eqnarray*}
There is a Central Limit Theorem for a sequence of independent, but not necessarily identically distributed
random variables. We work this out in the setting of a triangular array of Bernoulli random variables.
To control the rate of convergence,
we utilize a Berry--Esseen theorem (we refer to \cite{Ca15}, section 3).
%
\begin{theorem}\label{Berry:abstract}
Let $X_{n,k}$ for $1 \leq k \leq n$ be independent random variables 
with expected values $\mu_{n,k}$, variances $\sigma_{n,k}^2$ and
absolute third central moment 
$$\rho_{n,k} = \mathbb{E} (\vert X_{n,k} - \mu_{n,k} \vert^3) < \infty.$$ 
Let
$\mu_n = \sum_{k=1}^n \mu_{n,k}$, $\sigma_n^2:= \sum_{k=1}^n \sigma_{n,k}^2$ and $S_n:= \sum_{k=1}^n X_{n,k}$.
Let $S_n^{*}= (S_n - \mu_n)/\sigma_n$. Then
\begin{equation*}
{\norm{P(S_n^{*} < x) - \Phi(x)}}_{\mathbb{R}} \leq C \,\,
\frac{\sum_{k=1}^n \rho_{n,k}}{\sigma_n^3},
\end{equation*}
where $C >0$ is a universal constant. This constant can be chosen as $C=0.7975$ \cite{VB72}.
\end{theorem}
\subsection{Harper's method}
Let $n \in \mathbb{N}$.
Let $P_n(x)= \sum_{k=0}^n a_{n,k} \, x^k$ 
be a monomic polynomial of degree $n$, $ a_{n,k} \geq 0$ and $P_n(1)>0$.
Suppose the roots of $P_n(x)$ are real and $P_n(x) = \prod_{k=1}^n \left( x + r_k \right)$.
Harper \cite{Ha67} introduced a triangular array of Bernoulli random variables $X_{n,j}$ with
distribution 
\begin{equation*}
P(X_{n,j}=0):= \frac{r_j}{1 + r_j} \text{ and }
P(X_{n,j}=1):= \frac{1}{1 + r_j}.
\end{equation*}
Let $S_n:= \sum_{j=1}^n X_{n,j}$. Then $P(S_n=k) = \frac{a_{n,k}}{P_n(1)}$.
\begin{lemma} \label{estimate}
Let $X_{n,j}$ be given. Then
\begin{equation*} 
\mathbb{E}(X_{n,j}) = \frac{r_j}{1+r_j}, \,\, \, 
\mathbb{V}(X_{n,j}) = \frac{r_j}{( 1 + r_j)^2},\, \,\,
\mathbb{E} ( \vert X_{n,j} - \mathbb{E}(X_{n,j})\vert^3) =  \frac{r_j(1+ r_j^2)}{(1 + r_j)^4}.
\end{equation*}
\end{lemma}
This implies that $\mathbb{E} ( \vert X_{n,j} - \mathbb{E}(X_{n,j})\vert^3) < \mathbb{V}(X_{n,j})$.
\begin{proof}[Proof of Theorem \ref{BEapplication}] \ \\
Due to Lemma, \ref{estimate} we obtain
\begin{equation*}
\frac{\sum_{k=1}^n \rho_{n,k}}{\sigma_n^3} \leq \frac{\sum_{k=1}^n \sigma_{n,k}^2}{\sigma_n^3} = \frac{1}{\sigma_n}.
\end{equation*}
Finally, Theorem \ref{Berry:abstract} gives the result. 
\end{proof}
\subsection{Expected values $\mu_n$ and variances $\sigma_n^2$}
The family of polynomials $\{Q_n(x)\}_n$, where  $$Q_n(x)=\sum_{k=0}^n A_{n,k}\, x^k,$$
satisfy the three term recurrence 
relation (\ref{3term}). Then
$u_{n}=Q_{n}
\left( 1\right) $ satisfies the recurrence
relation 
\begin{equation}\label{eq:homogen1}
u_{n}-3u_{n-1}+u_{n-2}=0. \end{equation}

\begin{proof}[Proof of Lemma \ref{lem1}.]
The Fibonacci numbers
$F_{n}$ satisfy the recurrence relation
$F_{n}-F_{n-1}-F_{n-2}=0$. Therefore,
$F_{n-1}=F_{n}-F_{n-2}=F_{n+1}-2F_{n-1}+F_{n-3}$, which implies
$F_{n+1}-3F_{n-1}+F_{n-3}=0$. Therefore, the sequences constituted by the
$F_{2n}$ or $F_{2n+1}$ satisfy the same recurrence relation as $u_{n}$.
Now $u_{1}=1=F_{2}$ and $u_{2}=3=F_{4}$, so we obtain $u_{n}=F_{2n}$ for
$n\geq 1$.
\end{proof}
Next we prove the formulas for $\mu_n$ and $\sigma_n^2$.
\begin{proof}[Proof of Proposition \ref{prop1}.]
To determine $\mu_n$ and $\sigma_n$, we follow the strategy
offered by the standard method of solving linear non-homogeneous
difference equation (we refer to \cite{El05}, Section 2.4).
For $v_{n}=
Q_{n}
^{\prime }\left( 1\right) $ holds
$v_{n}-3v_{n-1}+v_{n-2}=u_{n-1}=F_{2n-2}$. The sequences constituted by
$F_{2n}$ or $F_{2n+1}$, resp., are a linearly independent solution of the
homogeneous difference equation (\ref{eq:homogen1}). 
Therefore, the solution
of the inhomogeneous difference equation
\begin{equation*}
v_{n}-3v_{n-1}+v_{n-2}=u_{n-1}.
\label{eq:inhomogen1}
\end{equation*}
is a linear
combination of $F_{2n}$, $nF_{2n}$, $F_{2n+1}$, $nF_{2n+1}$. From the initial
conditions $v_{1}=1$, $v_{2}=4$, $v_{3}=14$, and $v_{4}=46$,
we can determine the coefficients and obtain
$v_{n}=\frac{2}{5}nF_{2n+1}+\frac{2}{5}F_{2n}-\frac{1}{5}nF_{2n}$ for $n\geq 1$
and observe that this also holds for $n=0$.

Let $w_{n}=
Q_{n}
^{\prime \prime }\left( 1\right) $. Then
$w_{n}-3w_{n-1}+w_{n-2}=v_{n-1}$ and
$N\left( E\right) =\left( p\left( E\right) \right) ^{2}$ is an annihilator of
the right hand side. Since the sequences constituted by $F_{2n}$ and $F_{2n+1}$
are a fundamental system of solutions of (\ref{eq:homogen1}), we obtain a
fundamental system of solutions of the inhomogeneous one by $F_{2n}$,
$F_{2n+1}$, $nF_{2n}$, $nF_{2n+1}$, $n^{2}F_{2n}$, and $n^{2}F_{2n+1}$. The
coefficients can be determined from the values $w_{1}=0$, $w_{2}=2$,
$w_{3}=14$, $w_{4}=68$, $w_{5}=282$, and $w_{6}=1068$ and we obtain
$w_{n}=\left( \frac{1}{5}n^{2}-\frac{1}{25}n-\frac{8}{25}\right) F_{2n}+\frac{2}{25}n F_{2n+1}$.
Recall that
\begin{equation*}
\label{identitaet}F_{2n}^{2}+F_{2n}F_{2n+1}-F_{2n+1}^{2}=-1 \text{  for } n\geq 0.
\end{equation*}
This leads to the explicit formula for $\sigma_n^2$.
\end{proof}
\section{Location of the modes: proof of Theorem \ref{exact}}
\begin{proof}
Let
$A_{n,m}
=\binom{n+m-1}{2m-1}$.
Therefore,
$
A_{n,m}
-
A_{n,m+1}
=\binom{n+m-1}{2m-1}\left( 1-\frac{\left( n+m\right) \left( n-m\right) }{\left( 2m+1\right) 2m}
\right) $.
For the numerator of the expression in brackets, we obtain
$5m^{2}+2m-n^{2}=5\left( m+\frac{1}{5}\right) ^{2}-\frac{1}{5}-n^{2}$.
Therefore, $
A_{n,m}
>
A_{n,m+1}
$ for
$m
>
\frac{\sqrt{5n^{2}+1}-1}{5}$ and $
A_{n,m-1}
<
A_{n,m
}
$
for $m<\frac{\sqrt{5n^{2}+1}-1}{5}+1$.

Obviously, for $n^{2}=5m^{2}+2m$ we have
two modes. This equation is equivalent to
$5n^{2}=\left( 5m+1\right) ^{2}-1$. With $j=5m+1$, we obtain the Pell--Fermat
equation $j^{2}-5n^{2}=1$. All its non-negative solutions are $\left(
\begin{array}{c}
j \\
n
\end{array}
\right) =\left(
\begin{array}{cc}
9 & 20 \\
4 & 9
\end{array}
\right) ^{k^{\prime }}\left(
\begin{array}{c}
1 \\
0
\end{array}
\right) $ for $k^{\prime }\geq 0$. To be a solution to the original problem, the integer $j$ must satisfy
$j\equiv 1\mod 5$. Exactly even powers $k^{\prime }$ yield such a solution.
\end{proof}
\begin{proof}[Proof of Theorem \ref{CLT}]
To apply Theorem \ref{BEapplication}, we need to show that 
the variance $\sigma_n^2$ proceeds to infinity.
But this follows from Corollary \ref{limits}.
\end{proof}

\section{Local limit theorem: numerical data}
We consider Corollary \ref{local} for $k= a_n + b_n$.
Here $a_n = n \, a$ and $b_n = \sqrt{n} \,\, b$, where we approximated
$\mu_n$ by $a_n$ and $\sigma_n$ by $b_n$. This leads to
\begin{equation*}
\binom{n+ \floor{\frac{n}{\sqrt{5}}}-1}{
2 \, \floor{\frac{n}{\sqrt{5}}}-1} \,\, \sim \,\, 
\frac{ 5^{\frac{3}{4}}}{ 2 \sqrt{\pi}
\,\, \sqrt{n}} \, F_{2n}.
\end{equation*}
A local Berry--Esseen result
would suggest a rate of convergence by $1/ \, \sqrt{n}$.
Table \ref{table} gives some evidence.
\begin{table}[H]
\[
\begin{array}{rrr|rrr|rrr}
\hline
n&&&n&&&n&&\\ \hline \hline
2&0&1.4\cdot 10^{0}&20&1.7\cdot 10^{-1}&8.7\cdot 10^{-1}&200&6.6\cdot 10^{-2}&1.0\cdot 10^{-1}\\
3&3.7\cdot 10^{-1}&5.3\cdot 10^{-1}&30&1.6\cdot 10^{-1}&2.4\cdot 10^{-1}&300&5.4\cdot 10^{-2}&3.1\cdot 10^{-2}\\
4&1.9\cdot 10^{-1}&1.1\cdot 10^{0}&40&1.3\cdot 10^{-1}&5.8\cdot 10^{-1}&400&4.6\cdot 10^{-2}&1.9\cdot 10^{-1}\\
5&3.6\cdot 10^{-1}&3.0\cdot 10^{-1}&50&1.3\cdot 10^{-1}&1.6\cdot 10^{-1}&500&4.2\cdot 10^{-2}&9.9\cdot 10^{-2}\\
6&2.4\cdot 10^{-1}&9.0\cdot 10^{-1}&60&1.1\cdot 10^{-1}&4.4\cdot 10^{-1}&600&3.8\cdot 10^{-2}&4.2\cdot 10^{-2}\\
7&3.3\cdot 10^{-1}&1.6\cdot 10^{-1}&70&1.1\cdot 10^{-1}&1.1\cdot 10^{-1}&700&3.5\cdot 10^{-2}&1.0\cdot 10^{-2}\\
8&2.5\cdot 10^{-1}&6.6\cdot 10^{-1}&80&1.0\cdot 10^{-1}&3.4\cdot 10^{-1}&800&3.3\cdot 10^{-2}&1.1\cdot 10^{-1}\\
9&3.0\cdot 10^{-1}&7.5\cdot 10^{-2}&90&9.8\cdot 10^{-2}&8.1\cdot 10^{-2}&900&3.1\cdot 10^{-2}&5.6\cdot 10^{-2}\\
10&2.5\cdot 10^{-1}&4.7\cdot 10^{-1}&100&9.1\cdot 10^{-2}&2.8\cdot 10^{-1}&1000&2.9\cdot 10^{-2}&2.1\cdot 10^{-2}\\ \hline
\end{array}
\]
\caption{\label{table} 
Values of $\frac{1}{F_{2n}}\binom{n+\left\lfloor \frac{n}{\sqrt{5}}\right\rfloor -1}{2\left\lfloor \frac{n}{\sqrt{5}}\right\rfloor -1}$ (second columns) and $\left| \frac{2\sqrt{\pi }\sqrt{n}}{5^{\frac{3}{4}}F_{2n}}\binom{n+\left\lfloor \frac{n}{\sqrt{5}}\right\rfloor -1}{2\left\lfloor \frac{n}{\sqrt{5}}\right\rfloor -1}-1\right| \sqrt{n}$ (third columns).
}
\end{table}


\begin{thebibliography}{MMM99}
\bibitem[Be73]{Be73}E. Bender: \emph{Central and local limit theorems applied to asymptotic enumeration.}
\newblock 
J. Comb.\ Theory (A) \textbf{15}  (1973), 91--111.

\bibitem[Be96]{Be96} 
M. Benoumhani: \emph{Sur une propri\'{e}t\'{e} des polyn\^{o}mes 
\`{a} racines  r\'{e}elles n\'{e}gatives.}
\newblock
J. Math.\ Pures Appl.\ IX. S\'{e}r.\ \textbf{75}, Number 2 (1996), 85--110.

\bibitem[Be03]{Be03} 
M. Benoumhani: \emph{A sequence of binomial coefficients related to Lucas and Fibonacci numbers.}
\newblock
Journal of Integer Seq.\ 
\textbf{6} (2003), Article 03.2.1.




\bibitem[Ca75]{Ca75} E. R. Canfield: \emph{Asymptotic normality in binomial type enumeration.} Ph.~D.\
Thesis, University of California, San Diego, 1975.



\bibitem[Ca15]{Ca15}  E.~R. Canfield:
\emph{Asymptotic normality in enumeration.\/}
\newblock
In: Mikl{\'o}s B{\'o}na (ed.) Handbook of Enumeration, CRC Press. Discrete Mathematics and its Applications (2015), 255--280.


\bibitem[Da64]{Da64} J. N. Darroch: \emph{On the distribution of the number of successes in
independent trials.} 
\newblock Ann.\ Math.\ Statist.\ 
{\bf 35}, Number {
3} (1964), 1317--1321.

\bibitem[El05]{El05} S. Elaydi: \emph{Introduction to Difference Equations.\/}
\newblock
Undergraduate Texts in Mathematics. Springer, New York, third edition, 2005.



\bibitem[Fi11]{Fi11} H. Fischer: \emph{A History of the Central Limit Theorem.}
\newblock
From Classical to Modern Probability Theory. Springer, New York, NY (2011).



\bibitem[Ha67]{Ha67} L. Harper: \emph{Stirling behaviour is asymptotically normal.} 
\newblock
Ann. Math. Stat. \textbf{38} (1967),
410--414.



\bibitem[HNT20]{HNT20}
B. Heim, M. Neuhauser, R. Tr\"{o}ger: \emph{Zeros of recursively defined polynomials.}
\newblock J. Difference Equ. Appl. \textbf{26}, Number 4 (2020), 510--531.




\bibitem[Ko01]{Ko01} T. Koshy: \emph{Fibonacci and Lucas Numbers with Applications.}
\newblock Pure and Applied Mathematics. A Wiley--Interscience Series of Texts, Monographs, and Tracts. 2001.

\bibitem[MV59]{MV59} 
A. M. Morgan-Voyce: \emph{Ladder network analysis using Fibonacci numbers.\/}
IRE Trans.\ on Circuit
Theory, CT-\textbf{6} (Sept.\ 1959), 321--322.

\bibitem[TZ74]{TZ74}
S. Tanny, M. Zucker: \emph{On a unimodal sequence of binomial coefficients.}
Discrete Math.\ \textbf{9} (1974), 79--89.

\bibitem[TZ78]{TZ78}
S. Tanny, M. Zucker: \emph{Analytic methods applied to a sequence of binomial coefficients.}
Discrete Math.\ \textbf{24} (1978), 299--310.


\bibitem[VB72]{VB72} P. van Beek: \emph{An application of Fourier methods 
to the problem of sharpening the Berry--Esseen inequality.} 
Z. Wahrsch.\ Verw.\ Gebiete \textbf{23} (1972), 187--196.
\end{thebibliography}
\end{document}